\documentclass{amsart}

\usepackage[pdftex,hyperindex]{hyperref} %Links Inside Document
\newtheorem{theorem}{Theorem}[section]

\newtheorem{corollary}[theorem]{Corollary}

\theoremstyle{definition}
\newtheorem{definition}[theorem]{Definition}

\theoremstyle{remark}
\newtheorem{remark}[theorem]{Remark}

\numberwithin{equation}{section}

\usepackage[all,2cell]{xy}%To use xymatrix

\newcommand{\RS}{\text{R}\Sigma}

\begin{document}

\title[E-infinity coalgebra structure on chain complexes]{E-infinity coalgebra structure on chain complexes with coefficients in $\mathbb{Z}$}

\author[Jes\'us S\'anchez Guevara]{Jes\'us S\'anchez Guevara$^1$\\	
\hspace{0.05cm}\\
%$^1$Centro de Investigaci\'on en Matem\'atica Pura y Aplicada (CIMPA)\\
$^1$Center of Math Research and Applications (CIMPA-UCR),\\ 
University of Costa Rica, CR\\
\texttt{\href{mailto: jesus.sanchez_g@ucr.ac.cr}{jesus.sanchez\_g@ucr.ac.cr}}
}

%\address{Escuela de Matem\'aticas, Universidad de Costa Rica}
\address{Faculty of Mathematics, University of Costa Rica}
%\curraddr{Escuela de Matem\'aticas, Universidad de Costa Rica}
\email{jesus.sanchez\_g@ucr.ac.cr}
%\thanks{$^1$Centro de Investigaci\'on en Matem\'atica Pura y Aplicada CIMPA%}
%\date{%
%    $^1$Centro de Investigaci\'on en Matem\'atica Pura y Aplicada CIMPA\\%
%	\href{mailto: jesus.sanchez_g@ucr.ac.cr}{jesus.sanchez\_g@ucr.ac.cr}\\
%    \today
%}

%    General info

%CSM2020
%18M60 Operads (general)

%18M75 Topological and simplicial operads [See also 18N60]

%18N70 ∞-operads and higher algebra [See also 18M75]

%55U10 Simplicial sets and complexes in algebraic topol

%55U15 Chain complexes in algebraic topology


%\textbf{MSC-codes:}2010 
%18D50: Operads 
%55U15: Chain complexes

%\date{January 1, 2001 and, in revised form, June 22, 2001.}



\begin{abstract}
%The aim of this paper is to construct an $E_\infty$-operad $\mathcal{R}$ 
%and prove that this operad induces an $E_\infty$-coalgebra 
%structure on chain complexes with coefficients in $\mathbb{Z}$. 
%The operad $\mathcal{R}$ is an alternative to the description
%of the $E_\infty$-coalgebra structure on chain complexes 
%by the Barrat-Eccles operad. 
The aim of this paper is to construct an $E_\infty$-operad 
inducing an $E_\infty$-coalgebra structure on chain complexes with coefficients in $\mathbb{Z}$,
which is an alternative description to the $E_\infty$-coalgebra by the Barrat-Eccles operad. 
\end{abstract}

\maketitle

%\section*{This is an unnumbered first-level section head}
%This is an example of an unnumbered first-level heading.

%% The correct journal style for \specialsection is all uppercase; a known bug
%% in amsart.cls prevents this, so input must be uppercase until it is fixed.
%\specialsection*{This is a Special Section Head}
%\specialsection*{THIS IS A SPECIAL SECTION HEAD}
%This is an example of a special section head%
%%%%%%%%%%%%%%%%%%%%%%%%%%%%%%%%%%%%%%%%%%%%%%%%%%%%%%%%%%%%%%%%%%%%%%%%
%\footnote{Here is an example of a footnote. Notice that this footnote
%text is running on so that it can stand as an example of how a footnote
%with separate paragraphs should be written.
%\par
%And here is the beginning of the second paragraph.}%
%%%%%%%%%%%%%%%%%%%%%%%%%%%%%%%%%%%%%%%%%%%%%%%%%%%%%%%%%%%%%%%%%%%%%%%%
%.

%###############################################################################################################
%###############################################################################################################
\section{Introduction}

An $E_\infty$-coalgebra structure on chain complexes of simplicial sets 
with coefficients in $\mathbb{Z}$ is introduced by Smith in 
\cite{smith1994iterating}. 
He used an $E_\infty$-operad, denoted by $\mathfrak{S}$, with
each component $\RS_n$ a $\Sigma_n$-free bar resolution of $\mathbb{Z}$.
The morphisms $f_n:\RS_n \otimes C_*(X)\to C_*(X)^{\otimes n}$ determined by
the $E_\infty$-coalgebra structure contains a family of higher diagonals 
on $C_*(X)$, starting with an homotopic version of the 
iterated Alexander-Whitney diagonal. % (given by $x\mapsto f_n([\,]_n\otimes x)$).
The operad $\mathfrak{S}$ can be seen as a version of the Barratt-Eccles operad 
(see \cite{BARRATT197425}), which is used by
Berger and Fresse in \cite{berger-fresse2002} to construct an explicit coaction on
normalized chain complexes extending the structure given by the Alexander-Whitney diagonal.

%we review the construction of the $E_\infty$-operad $\mathfrak{S}$
%given by Smith in \cite{math/0004003v7}\footnote[2]{An updated version of \cite{smith1994iterating}} 

%and his proof that $C_*(X)$ is an $E_\infty$-coalgebra using this operad. 
%In this article, an $E_\infty$-structure on the chain complex of an simplicial set 
%it is constructed an $E_\infty$-operad $\mathcal{R}$  
%Next, we give an alternative proof of the $E_\infty$-structure on the
%chain complex of an simplicial set by using an operad $\mathcal{R}$ constructed
%by us that simplify the task. The method used to construct $\mathcal{R}$ 
%gives an simply way to produce $E_\infty$-operads.

In this paper we present a new operad $E_\infty$-operad $\mathcal{R}$  
inducing an $E_\infty$-coalgebra structure on chain complexes,
which is designed following the ideas by Smith in his construction of $\mathfrak{S}$,
and we show in section $3$ that $\mathfrak{S}$ can be obtained 
from $\mathcal{R}$ by operadic quotients.
 
%In this article, it is constructed an $E_\infty$-operad $\mathcal{R}$  
%which is used to give an alternative description of the $E_\infty$-structure
%on the chain complex of an simplicial set.
%The method used to construct $\mathcal{R}$ gives 
%an simply way to produce $E_\infty$-operads.

It is worth pointing out that $\mathcal{R}$ 
presents similarities with 
the bar-cobar resolution of Ginzburg-Kapranov (see \cite{2007arXiv0709.1228G}),
and Berger and Moerdij in \cite{Iekecolored2005}
identified this resolution with the $W$-construction of Boardman and Vogt (see \cite{boardman1973homotopy}).
As a consequence, $W$-construction of the Barratt-Eccles operad gives a cofibrant resolution of it.
%given as a result that applied to the Barratt-Eccles operad, 
Then, our operad $\mathcal{R}$ 
may be seen as a middle point between the Barratt-Eccles operad and its $W$-construction.

%The ground category in this chapter is DGA-$\mathbb{Z}$-Mod. To simplify the notation
%it will be written DGA-Mod. All the operads are operads on DGA-Mod.

The results in this paper are based on the author's PhD thesis \cite{sanchez-jesus-thesis-2016},
where $E_\infty$-coalgebras are identified over structures associated to chain complexes,
which are generalizations of uniqueness properties described by Prout\'e in \cite{alain-transf-EM-1983} 
and \cite{alain-transf-AW-1984} of the Eilenberg-Mac lane transformation.

%determinated by the Eilenberg-Mac lane transformation.

%used to homotopy properties,
%described by Alain Prout\'e in \cite{alain-transf-EM-1983} and \cite{alain-transf-AW-1984},
%of structures associated to chain complexes determinated by the Eilenberg-Mac lane transformation.

%where is the construction of $E_\infty$-operads is needed to study homotopy properties,
%described by Alain Prout\'e in \cite{alain-transf-EM-1983} and \cite{alain-transf-AW-1984},
%of structures associated to chain complexes determinated by the Eilenberg-Mac lane transformation.

\section{Preliminaries}
\label{sec-preliminaries}

\subsection{Differential graded modules}
\label{subsec-the-category-dga}

A $\mathbb{Z}$-module $M$ is graded 
if there is a collection $\{M_i\}_{i\in \mathbb{Z}}$
of submodules of $M$ such that $M=\bigoplus_{i\in \mathbb{Z}}M_i$.
A differential graded module with augmentation and coefficients in $\mathbb{Z}$,
or $DGA$-module for short,
is a graded $\mathbb{Z}$-module $M$ together with 
a morphism $\partial:M\to M$ of degree $-1$
	%NOTE: morphism =  graded $\mathbb{Z}$-module morphism, but is too much redundant
such that $\partial^2=0$, 
and morphisms $\epsilon:M\to \mathbb{Z}$, $\eta:\mathbb{Z}\to M$ of degree $0$, 
	%NOTE: morphism =  graded $\mathbb{Z}$-module morphism, but is too much redundant
called augmentation and coaugmentation of $M$, respectively,  
such that $\epsilon \circ \eta=\text{id}$. 
	%NOTE: use \text{id} instead of 1_{\mathbb{Z}}
The category of $DGA$-modules is denoted $DGA\text{-Mod}$.

\subsection{Operads}
\label{subsec-operads}

An operad $P$ in the monoidal category $DGA\text{-Mod}$ 
is a collection of $DGA$-modules $\{P(n)\}_{n\geq 1}$ together with 
right actions of the symmetric group $\Sigma_n$ on each component $P(n)$, 
	%NOTE: action of symmetric group ON? each component? R/ YES
	% right group action of G on X is 
and morphisms of the form $\gamma:P(r)\otimes P(i_1)\otimes P(i_r)\to P(i_1+\cdots+i_r)$,
which satisfy the usual conditions of existence of an unit, associativity and equivariance.
The morphisms $\gamma$ will be called composition morphisms of the operad.
A morphism between operads $f:P\to Q$, is 
a collection of $DGA$-morphisms $f_n:P(n)\to Q(n)$ of degree $0$,
respecting units, composition and equivariance. 
The category of operads is denoted $\mathcal{OP}$

If we forget composition morphisms of an operad $P$, 
the collection of $DGA$-modules with right actions that remains
is called an $\mathbb{S}$-module.
They form a category denoted $\mathbb{S}\text{-Mod}$. 
The forgetful functor $U:\mathcal{OP}\to \mathbb{S}\text{-Mod}$
has a right adjoint denoted $F:\mathbb{S}\text{-Mod}\to \mathcal{OP}$, 
called the free operad functor.

\begin{definition}\label{df-ideal-operad}
Let $\mathcal{P}$ be an operad on the category $DGA\text{-Mod}$, 
with composition $\gamma$. 
A sub $\mathbb{S}$-module $\mathcal{I}$ of $U(\mathcal{P})$ 
is called an operadic ideal of $\mathcal{P}$ if it satisfies 
$\gamma(x\otimes y_1\otimes \cdots \otimes y_k)\in \mathcal{I}$,
whenever some elements $x,y_1,\ldots,y_k$ belongs to $\mathcal{I}$.

\end{definition}

\begin{definition}\label{def-op-ideal-quotient}
Let $\mathcal{P}$ be an operad and $\mathcal{I}$ an operadic ideal of $\mathcal{P}$.
We define the quotient operad $\mathcal{P}/\mathcal{I}$ as the operad  
given by $(\mathcal{P}/\mathcal{I})(n)=P(n)/I(n)$ 
for every $n\geq 1$, and composition induced by the composition of $\mathcal{P}$.
\end{definition}

\begin{remark}
Clearly, the operad structure of $\mathcal{P}/\mathcal{I}$ is well defined
by the properties of the operadic ideal $\mathcal{P}$, 
which allow passing to the quotient the composition of $\mathcal{P}$
(see \cite{2007arXiv0709.1228G}~$\S$2.1).
\end{remark}

\subsection{The Bar Resolution}
\label{subsec-the-bar-resolution}
%$\Sigma_n$ will denote the symmetric group of the set $[n]=\{1,\ldots,n\}$.
	%NOTE: Sigma_n appear in the section before.
The chain complex with coefficients in $\mathbb{Z}$ given by 
the $\Sigma_n$-free bar resolution of $\mathbb{Z}$ 
is denoted $R\Sigma_n$. 
Recall that degree $m$ elements of $R\Sigma_n$ are 
$\mathbb{Z}$-linear combinations of elements of the form $\sigma[\sigma_1/\cdots/\sigma_m]$, 
where $\sigma,\sigma_1,\ldots,\sigma_m \in \Sigma_n$
and border 
$\partial=\sum_{i=0}^{m}(-1)^i\partial_i$, where 
$\partial_0[\sigma_1/\cdots/\sigma_m]=\sigma_1[\sigma_2/$
$\cdots/\sigma_m]$,
for $0<i<m$, 
$\partial_i[\sigma_1/\cdots/\sigma_m]=[\sigma_1/\cdots/\sigma_i\sigma_{i+1}/\cdots/\sigma_m]$,
and $\partial_m[\sigma_1/$
$\cdots/\sigma_m]=[\sigma_1/\cdots/\sigma_{m-1}]$.
In  zero degree, the $\mathbb{Z}[\Sigma_n]$-module 
	%NOTE: I'm not sure of using "$\mathbb{Z}[\Sigma_n]$-module", because it is never defined before.
is generated by one element, written $[\,]$.
$R\Sigma_n$ is acyclic with
contracting chain homotopy 
the map $\psi_n:R\Sigma_n\to R\Sigma_n$ 
of degree $1$ defined by the relations $\psi_n [\sigma_1/\cdots/\sigma_m]=0$ and
$\psi_n \sigma[\sigma_1/\cdots/\sigma_m]=[\sigma/\sigma_1/\cdots/\sigma_m]$.

\subsection{$E_\infty$-Operads}

%In the last chapter we defined $E_\infty$-operads for the case of 
%differential graded modules with coefficients in $\mathbb{Z}$.
%For coefficients in a field $\field$, we take the obvious adaptation.

\begin{definition}%[$E_\infty$-Operad] 
\label{df-Einf-Operad}
\index{einfinityperad@$E_\infty$-operad}
An operad $\mathcal{P}$ in the category $DGA$-Mod is called $E_\infty$-operad 
if each component $P(n)$ is a $\Sigma_n$-free resolution of $\mathbb{Z}$.
\end{definition}

\begin{definition}%[$E_\infty$-algebra and $E_\infty$-coalgebra] 
\label{df-Einf-Algebra-Coalgebra}
\index{einfinityalgebra@$E_\infty$-algebra}
\index{einfinitycoalgebra@$E_\infty$-coalgebra}
We call $E_\infty$-coalgebra(algebra) any $\mathcal{P}$-coalgebra(algebra) with $\mathcal{P}$ an
$E_\infty$-operad. 
%Similarly, an $E_\infty$-coalgebra is an $\mathcal{P}$-coalgebra where the operad $\mathcal{P}$ is an $E_\infty$-operad.
\end{definition}

We introduce a notion of morphism between $E_\infty$-coalgebras which is well suited for our purpose.

\begin{definition}
Let $\mathcal{P}$ be an $E_\infty$-operad in the category $DGA$-Mod, 
and let $A,B$ $\mathcal{P}$-coalgebras. 
A morphism $f:A\to B$ of $\mathcal{P}$-coalgebras
is a morphism of $DGA$-Mod which preserves the $\mathcal{P}$-coalgebra
structure up to homotopy, that is, the following diagram:

\begin{equation}
\begin{gathered}
\xymatrix{
\mathcal{P}(n)\otimes A\ar[r]^-{\varphi_n^A} \ar[d]_-{1\otimes f}& A^{\otimes n} \ar[d]^-{f^{\otimes n}}\\
\mathcal{P}(n)\otimes B \ar[r]_-{\varphi_n^B} & B^{\otimes n}
}
\end{gathered}
\end{equation}

is commutative up to homotopy for every $n>0$, where $\varphi^A_n$ and $\varphi^B_n$
are the associated morphisms of the $\mathcal{P}$-coalgebra structure of $A$ and $B$,
respectively. The category of $\mathcal{P}$-coalgebras is denoted
$\mathcal{P}\text{-CoAlg}$.
\end{definition}
	%NOTE in our definition both E-inf coalgebra must be
	%P-coalgebra with the same operad.

%###############################################################################################################
%###############################################################################################################

\section{The Operad $\mathcal{R}$}\label{sec-the-operad-R}

In this section is constructed an $E_\infty$-operad $\mathcal{R}$ 
which is used to describe 
the complex $C_*(X)$ as an $E_\infty$-coalgebra. 
%to the operad $\mathfrak{S}$ given by Smith \cite{smith1994iterating}.

%Roughly speaking, to construct the operad $\mathcal{R}$, 
%first take the $\mathbb{S}$-module with components 
%the $\mathbb{Z}[\Sigma_n]$-free bar resolutions of $\mathbb{Z}$, 
%and then make the quotient of the free operad on this $\mathbb{S}$-module 
%by a suitable operad ideal $\mathcal{I}$ ,
%such that our operad will have only one generator of degree $0$ in each component. 

\begin{definition}\label{prop-ideal-operad-R}
Let $S$ be the $\mathbb{S}$-module in the category $DGA\text{-Mod}$,
with components $S(n)=\text{R}\Sigma_n$, 
the $\mathbb{Z}[\Sigma_n]$-free bar resolution of $\mathbb{Z}$.
Define the operad $\mathcal{R}$ as the quotient operad $F(S)/\mathcal{J}$,
where $\mathcal{J}$ is the operadic ideal of the free operad $F(S)$
generated by the elements of zero degree of $F(S)$ of the form $x-y$,
where $x$ and $y$ are not null. 
\end{definition}

\begin{theorem}\label{prop-R-Einf-op}
The operad $\mathcal{R}$ is an $E_\infty$-operad and induces an $E_\infty$-coalgebra structure on $C_*(X)$.
\end{theorem}

\begin{proof}
It suffices to exhibit in each arity an contracting chain homotopy. 
In arity $n$, the contracting chain homotopy $\Phi_n:R(n)\to R(n)$ is
obtained by extending on $R(n)$ the contracting chain homotopy $\psi_n$
from the component $\text{R}\Sigma_n$ of $S$ as follows. 

$R(2)$ is isomorphic to $S(2)$, so the contracting chain homotopy remains the same.
When $n>2$, $R(n)$ has two types of elements: 
the elements from the injection $S(n)\to R(n)$
and the elements of the form $\gamma(x;y_1,\ldots,y_r)$, 
where $x\in S(r)$ and $y_j \in R(i_j)$. 
In the first case $\Phi_n$ will behaves as the contracting chain homotopy in $S(n)$, 
and for the second case, we define
$\Phi_n \gamma(x;y_1,\ldots,y_r)=\gamma(\Phi_n(x);y_1,\ldots,y_r)$.

%Each component $\mathcal{R}$ has a contracting chain homotopy, that is,
%for every $R(m)$ there is an application $\theta_m$ of degree $1$ such
%that $\partial \theta_m+\theta_m \partial=1$.

%The contracting chain homotopy $\theta_m$ is induced by the on of $S(m)$. It suffices to
%check the property for the elements of $R(m)$ that comes from formal compositions,
%in other words, of the form $\gamma(z;z_1,\ldots,z_r)$, where $z\in S(r)$, $z_i\in S(r_i)$
%and such that $r_1+\cdots+r_r=m$.

To check that $\partial \Phi_n+\Phi_n \partial=1$,
let $x$ of the form $[\sigma_1|\cdots |\sigma_l]$, with $\sigma_j\in \Sigma_r$. 
Now $\partial \Phi_n\gamma(x;y_1,\ldots,y_r)=\partial \gamma(\Phi_n (x);y_1,\ldots,y_r)=0$.
On the other hand, 

\begin{align}
&\Phi_n \partial \gamma(x;y_1,\ldots,y_r) \\ 
=&\Phi_n\gamma(\partial x;y_1,\ldots,y_r)+
(\text{sign})\sum \Phi_n \gamma(x;y_1,\ldots,\partial y_j,\ldots,y_r)\\
=&\gamma(\Phi_n\partial x;y_1,\ldots,y_r)+
(\text{sign})\sum \gamma(\Phi_n x;y_1,\ldots,\partial y_j,\ldots,y_r)\\
=&\gamma(x-\partial\Phi_n x;y_1,\ldots,y_r)\\
=&\gamma(x;y_1,\ldots,y_r)
\end{align}

When $x$ has the form $\sigma[\sigma_1|\cdots |\sigma_l]$ the verification is similar, 
because compositions $\gamma$ equivariance:
$\gamma(\sigma [\sigma_1|\cdots |\sigma_l];y_1,\ldots,y_r)=\gamma([\sigma_1|\cdots |\sigma_l];y_{\sigma^{-1}(1)},\ldots,y_{\sigma^{-1}(l)})$.

Now, the universal property of the coaugmentation $\iota$ of the adjunction $F\vdash U$,
gives the commutative diagram:

\begin{equation}
\begin{gathered}
\xymatrix{
S \ar[rd]_-{i} \ar[r]^-{\iota} & F(S) \ar[d]^-{p}\\
& \mathfrak{S}
}
\end{gathered}
\end{equation}

Where the morphism $i$ is the identity of $\mathbb{S}$-modules. 
%The morphism of operads $p:F(S)\to \mathfrak{S}$ is given by the universal property of $\epsilon$.
It is easy to see that $p$ respect the ideal $\mathcal{J}$ 
because, when the free operad construction is 
interpreted by rooted trees, $p$ is essentially the contraction of vertices of trees.
Thus $p$ pass to the quotient and we obtain a morphism of operads
$\overline{p}:\mathcal{R}\to \mathfrak{S}$, which implies that
every $\mathfrak{S}$-coalgebra is an $\mathcal{R}$-coalgebra.
\end{proof}

\begin{corollary}
The construction in theorem \ref{prop-R-Einf-op} is functorial.
\end{corollary}

\begin{proof}
The functoriality of the $\mathfrak{S}$-coalgebra structure is inherited by the $\mathcal{R}$-coalgebra structure
by the operad morphism $\overline{p}:\mathcal{R}\to \mathfrak{S}$ in the proof of theorem \ref{prop-R-Einf-op},
as shows the following commutative diagram for every morphism $f:X\to Y$:
\begin{equation}
\begin{gathered}
\xymatrix{
\mathcal{R} \ar[r]^-{\overline{p}} & \mathfrak{S}\ar[r] \ar[rd]& \text{CoEnd}(C_*(X)) \ar[d]^-{f_*}\\
				  & 	 &\text{CoEnd}(C_*(Y))
}
\end{gathered}
\end{equation}
\end{proof}

We can understand the relation between operad $\mathcal{R}$ and operad $\mathfrak{S}$
with the following proposition.

\begin{corollary}
There is an operad ideal $\mathcal{I}$ such that $\mathfrak{S}\cong \mathcal{R}/\mathcal{I}$.
\end{corollary}

\begin{proof}
This is because the underlying $\mathbb{S}$-module of $\mathfrak{S}$ is $S$,
and a direct consequence of the definition of compositions $\gamma$ of $\mathfrak{S}$(see \cite{smith1994iterating}),
in the sense that, the operadic ideal $\mathcal{I}$ is defined by the
identification needed for $\gamma$. 
%\mathfrak{S}By induction we have to make the identifications $x\circ_i y=\varphi(\partial x \circ_i y +(-1)^k x\circ_i \partial y)$.
\end{proof}

In \cite{2015arXiv150302701D} Vallette and Dehling describe 
an operad similar to $\mathcal{R}$ 
and 
%they show that this operad can be used to explicitly 
state (by the use relations) a definition of $E_\infty$-algebras.
In this sense, $\mathcal{R}$-coalgebras can be 
described as follows.

%with operad $\mathcal{R}$ , 
%as it is already possible for $A_\infty$-algebras.

\begin{corollary}
Let $A$ be a $DGA$-module together with:

\begin{enumerate}
\item For every integer $m\geq 1$, $n\geq 1$ and $\sigma,\sigma_1,\ldots,\sigma_n \in \Sigma_m$,
morphisms of degree $n$:
	\[
		\mu_{\sigma[\sigma_1/\cdots/\sigma_n]_m}:A\to A^{\otimes n}.
	\]
\item For every integer $m\geq 1$ and $\sigma\in \Sigma_m$, applications of degree $0$:
	\[
		\mu_{\sigma[\,\,]_m}:A\to A^{\otimes n}.
	\]
\end{enumerate}

Suppose these morphisms satisfy the following relations:
\begin{enumerate}
\item $\mu_{\sigma x}=\mu_{x}\sigma$, where $\sigma$ is the right action on $n$ factors.
\item $\mu_{x+y}=\mu_{x}+\mu{y}$ and $\partial \mu_{x}=\mu_{\partial x}$.
\item $(\mu_{[\,\,]_{m_1}}\otimes \cdots \otimes \mu_{[\,\,]_{m_n}})\mu_{[\,\,]_n}=\mu_{[\,\,]_{m_1+\cdots +m_n}}$.
\end{enumerate}

Then, $A$ is an $\mathcal{R}$-coalgebra. The converse is also true.
\end{corollary}

\begin{proof}
This is directly implied by the operad morphism $\mathcal{R}\to \text{Coend}(A)$.
\end{proof}

\section{Funding}

This work was supported by the University of Costa Rica [grant number OAICE-08-CAB-144-2010].
I thank Alain Prout\'e for encouraging me to think about this problem.

\bibliographystyle{amsplain}

\bibliography{principal-bibliography}

%\begin{thebibliography}{10}

%\bibitem {A} T. Aoki, \textit{Calcul exponentiel des op\'erateurs
%microdifferentiels d'ordre infini.} I, Ann. Inst. Fourier (Grenoble)
%\textbf{33} (1983), 227--250.

%\bibitem {B} R. Brown, \textit{On a conjecture of Dirichlet},
%Amer. Math. Soc., Providence, RI, 1993.

%\bibitem {D} R. A. DeVore, \textit{Approximation of functions},
%Proc. Sympos. Appl. Math., vol. 36,
%Amer. Math. Soc., Providence, RI, 1986, pp. 34--56.

%\end{thebibliography}

\end{document}